\documentclass[a4paper]{amsart}
\usepackage{amsmath,amssymb,amsxtra,amsfonts,amsthm,comment,url,microtype}

\usepackage{tikz} 
\usetikzlibrary{arrows,backgrounds,decorations,automata}

\theoremstyle{plain}
\newtheorem{theorem}{Theorem}
\newtheorem{lemma}{Lemma}

\theoremstyle{definition}

\newcommand{\B}{\mathbb}

\newcommand{\gd}{\delta}

\newcommand{\gD}{\Delta}

\begin{document}

\title{Zero order estimates for Mahler functions}

\author{Michael Coons}
\address{School of Math.~and Phys.~Sciences\\
University of Newcastle\\
Callaghan\\
Australia}
\email{Michael.Coons@newcastle.edu.au}

\thanks{The research of M.~Coons was supported by ARC grant DE140100223.}

\date{\today}


\begin{abstract} We give an upper bound for the zero order of the difference between a Mahler function and an algebraic function. This complements estimates of Nesterenko, Nishioka, and T\"opfer, among others, who considered polynomials evaluated at Mahler functions.
\end{abstract}

\maketitle

\section{Introduction}

When Kurt Mahler was quite young and sick in bed, he set out to prove that the number $\sum_{n\geqslant 0}2^{-2^n}$ is transcendental. He did so by exploiting the functional equation $$F(z^2)=F(z)-z,$$ which the series $\sum_{n\geqslant 0}z^{2^n}$ satisfies. In doing so, Mahler discovered an important method in the theory of Diophantine approximations with applications to transcendence and algebraic independence. 

Mahler's results were generalised and extended by several authors. The most general of generalisations, which seems to capture all previous versions considered, was given by T\"opfer \cite{T1995, T1998}, who considered formal power series $f_1(z),\ldots,f_d(z)\in\B{C}[[z]]$ that satisfy functional equations $$A_0(z,f_1(z),\ldots,f_d(z))\cdot f_i(T(z))=A_i(z,f_1(z),\ldots,f_d(z))\quad (1\leqslant i\leqslant d),$$ where $T(z)\in\B{C}(z)$ and $A_i(z,y_1,\ldots,y_d)\in\B{C}[z,y_1,\ldots,y_d]$ for $i=0,\ldots,d$. For this version, T\"opfer produced both a zero order estimate \cite{T1998} for $Q(z,f_1(z),\ldots,f_d(z))$ with $Q(z,y_1,\ldots,y_d)\in\B{C}[z,y_1,\ldots,y_d]$ as well as algebraic independence results for some special cases of his generalisation \cite{T1995}. 

Of all of the generalisations, two stand out and are arguably the most important; they are also the simplest. The first was given by Mahler himself \cite{M1929}, who considered\footnote{Actually, Mahler only considered the base field $\overline{\B{Q}}$ and not all of $\B{C}$.} functions $f(z)\in\B{C}[[z]]$ satisfying $$f(z^k)= R(z,f(z)),$$ for an integer $k\geqslant 2$ and a rational function $R(z,y)\in\B{C}(z,y)$. The second is harder to attribute, but goes back at least to  the 1960s or 1970s. In this case, one considers a function $f(z)$ for which there are integers $k\geqslant 2$ and $d\geqslant 1$ such that \begin{equation}\label{MFE}a_0(z)f(z)+a_1(z)f(z^k)+\cdots+a_d(z)f(z^{k^d})=0,\end{equation} for some polynomials $a_0(z),\ldots,a_d(z)\in\B{C}[z]$. These two generalisations coincide when $d=1$, and for this value of $d$ the strongest results have been shown. While there are few natural examples of the other classes, functions satisfying \eqref{MFE} are readily available and certain cases are of particular importance in theoretical computer science; the generating functions of automatic and regular sequences satisfy \eqref{MFE}. See the works of Allouche and Shallit \cite{AS1992,ASbook,AS2003}, Christol, Kamae, Mend{\`e}s France, and Rauzy \cite{CKMR1980}, Dekking, Mend{\`e}s France, and van der Poorten \cite{DMP1982}, Loxton \cite{L1988}, and Becker \cite{pgB1994} for further details and specific examples.

In this paper, we are concerned with the algebraic approximation of functions satisfying \eqref{MFE}. We call a function satisfying \eqref{MFE} a {\em $k$-Mahler function}, or just a {\em Mahler function} when $k$ is clear. The minimal such $d$ for which \eqref{MFE} holds for $f(z)$ is called the {\em degree} of the Mahler function $f(z)$, denoted $d_f$, and we define the {\em height} of the Mahler function $f(z)$ by $A_f:=\max\{\deg a_i(z):i=0,\ldots,d_f\}$.

Our main result is a zero order estimate for the difference of a Mahler function with an algebraic function. To this end, let $\nu:\B{C}((z))\to\B{Z}\cup\{\infty\}$ be the valuation defined by $\nu(0):=\infty$ and $$\nu\left(\sum c_n z^n \right):=\min\{i:c_i\neq 0\}$$ when $\sum_{n} c_n z^n$ is nonzero. Also, for $G(z)$ an algebraic function with minimal polynomial $P(z,y)\in\B{C}[z,y]$, we call $\deg_y P(z,y)$ the {\em degree} of $G(z)$ and we call $\exp\left(\deg_z P(z,y)\right)$ the height of $G(z)$.

\begin{theorem}\label{main} If $F(z)$ is an irrational $k$-Mahler function  of degree $d_F$ and height $A_F$, and $G(z)$ is an algebraic function of degree at most $n$ and height at most $H_G$, then $$\nu\big(F(z)-G(z)\big)\leqslant (d_F+1)\cdot A_F\cdot n^{d_F+1}+\frac{k^{d_F+1}-1}{k-1}\cdot\log H_G\cdot n^{d_F}.$$
\end{theorem}

Previous results on zero estimates of Mahler functions focussed on upper bounds for $\nu(Q(z,F(z)))$ for polynomials $Q(z,y)\in\B{C}[z,y]$ and used quite deep methods, relying on Nesterenko's elimination-theoretic method \cite{Nes1977,Nes1984}; see also Becker \cite{B1991}, Nishioka \cite{N1991}, and T\"opfer \cite{T1998}. While the estimate provided by Theorem \ref{main} is essentially of the same order as the best bounds for $\nu(Q(z,F(z)))$, our proof is much simpler---it avoids the use of Nesterenko's and Nishioka's methods---and is by all means, elementary. 

\section{Algebraic approximation of Mahler functions}

In recent work with Jason Bell \cite{BC2016}, we proved the following result.

\begin{lemma}[Bell and Coons]\label{BC2016} Let $F(z)$ be an irrational $k$-Mahler function of degree $d_F$ and height $A_F$, and let $P(z)/Q(z)$ be any rational function with $Q(0)\neq 0$. Then $$\nu\left(F(z)-\frac{P(z)}{Q(z)}\right)\leqslant A_F+\frac{k^{d_F+1}-1}{k-1}\cdot\max\{\deg P(z), \deg Q(z)\}.$$
\end{lemma}

Theorem \ref{main} is the generalisation of this result to approximation by algebraic functions. To prove this generalisation, we use a resultant argument.

\begin{lemma}\label{degrees} Let $f(z)$ and $g(z)$ be two algebraic functions of degrees at least $2$ satisfying polynomials of degrees $\gD_f$ and $\gD_g$ with coefficients of degree at most $\gd_f$ and $\gd_g$, respectively. Then the algebraic function $f(z)+g(z)$ satisfies a polynomial of degree $$\gD_{f+g}\leqslant \gD_f\gD_g$$ with coefficients of degree $$\gd_{f+g}\leqslant \gd_f\gD_g+\gd_g\gD_f.$$
\end{lemma}

\begin{proof} This result follows by using the Sylvester matrix to calculate a certain resultant. For $R$ a ring and $P,Q\in R[y]$ with $$P(y)=\sum_{i=0}^{\deg_y P} p_i y^i\quad\mbox{and}\quad Q(y)=\sum_{i=0}^{\deg_y Q} q_i y^i,$$ the resultant of $P$ and $Q$ with respect to the variable $y$ is denoted by ${\rm res}_y(P,Q)$ and may be calculated as the determinant of the $(\deg_y Q+\deg_y P)\times (\deg_y Q+\deg_y P)$ Sylvester matrix; that is $${\rm res}_y(P,Q):=\det\left(\begin{matrix}
p_0 & p_1 & p_2 &\cdots &p_{\deg_y P} & & & \\ 
 & p_0 & p_1 & p_2 &\cdots &p_{\deg_y P} & & \\
 & & \ddots & \ddots & \ddots & &\ddots  & \\ 
 & & & p_0 & p_1 & p_2 &\cdots &p_{\deg_y P} \\
q_0 & q_1 & q_2 &\cdots &q_{\deg_y Q} & & & \\ 
 & q_0 & q_1 & q_2 &\cdots &q_{\deg_y Q} & & \\
 & & \ddots & \ddots & \ddots & &\ddots  & \\ 
 & & & q_0 & q_1 & q_2 &\cdots &q_{\deg_y Q} \\ 
\end{matrix}\right),$$ where there are $\deg_y Q$ rows of the coefficients of $P$ and $\deg_y P$ rows of the coefficients of $Q$. Now suppose $R=\B{C}[z,x]$, so that the entries of the above Sylvester matrix are polynomials in the variables $z$ and $x$, and set $D(x,z):={\rm res}_y(P,Q)$. Since polynomial degrees are additive, using the Leibniz formula for the determinant, we have immediately that \begin{equation}\label{degzD}\deg_z D(x,z)\leqslant \deg_y Q \deg_z P+\deg_y P \deg_z Q\end{equation} and \begin{equation}\label{degxD}\deg_x D(x,z)\leqslant \deg_y Q \deg_x P+\deg_y P \deg_x Q.\end{equation}

The lemma now follows immediately by combining \eqref{degzD} and \eqref{degxD} with the fact that given algebraic functions $f(z),g(z)\in\B{C}[[z]]$ and polynomials $P_f(z,y),P_g(z,y)\in\B{C}[z,y]$ with $P_f(z,f)=P_g(z,g)=0$, the algebraic function $f(z)+g(z)$ is a root of the polynomial ${\rm res}_y(P_f(z,y),P_g(z,x-y))$ viewed as a polynomial in $x$.
\end{proof}

Because of Lemma \ref{BC2016}, we may focus on algebraic functions of degree at least $2$.

\begin{lemma}\label{MGprops} Let $a_0(z),\ldots,a_d(z)$ be polynomials of degree at most $A$. If $G(z)\in\B{C}[[z]]$ is an algebraic function of degree $\gD_G\geqslant 2$ satisfying a minimal polynomial with coefficients of degree at most $\gd_g$, then the function $$M_G(z):=\sum_{i=0}^d a_i(z)G(z^{k^i})$$ is an algebraic function satisfying a polynomial of degree $$\gD_{M_G}\leqslant \gD_G^{d+1}$$ whose coefficients have degree $$\gd_{M_G}\leqslant (d+1)A\cdot\gD_G^{d+1}+\frac{k^{d+1}-1}{k-1}\cdot\gd_G\cdot\gD_G^{d}.$$
\end{lemma}

\begin{proof} Since $G(z)$ is an algebraic function, so is $\sum_{i=0}^d a_i(z)G(z^{k^i})$. One can easily gain information about the sum using the theory of resultants. 

To get an upper bound on $\nu(M_G(z))$, we apply the idea of the previous paragraph by including the terms $G_i(z):=a_i(z)G(z^{k^i})$ one at a time. To do this, let $$P_{G}(z,y):=g_{\gD_{G}}y^{\gD_{G}}+\cdots+g_1y+g_0$$ be the minimal polynomial of $G(z)$. Here we have denoted the degree of $G(z)$ by $\Delta_G$. Set $\gd_G:=\deg_z P_G(z,y)$. Then $$P_{G_i}(z,y)=a_i(z)^{\Delta_G}P_{G}(z^{k^i},y/a_i(z))$$ is a polynomial with $P_{G_i}(z,G_i(z))=0$, where, of course, we only form this polynomial when $a_i(z)\neq 0$. Here, we have that $P_{G_i}(z,y)$ is still minimal with respect to the degree of $y$, but there is no guarantee that it is minimal with respect to the degree of $z$ for this degree of $y$. However, we do have that the minimal polynomial of $G_i(z)$ divides $P_{G_i}(z,y)$ and the remainder is just a polynomial in $z$. In any case, the above gives that \begin{equation}\label{DGi}\Delta_{G_i}:=\deg_y P_{G_i}(z,y)=\deg_y P_G(z,y)=\Delta_G\end{equation} and \begin{equation}\label{dGi}\delta_{G_i}:=\deg_z P_{G_i}(z,y)\leqslant A\Delta_G+k^i\delta_G.\end{equation}
The lemma now follows by combining \eqref{DGi} and \eqref{dGi} with Lemma \ref{degrees}.
\end{proof}

\begin{lemma}\label{nuG} Let $G(z)\in\B{C}[[z]]$ be an algebraic function of degree at least $2$ satisfying the polynomial $P_G(z,y)=a_n(z)y^n+a_{n-1}(z)y^{n-1}+\cdots+a_1(z)y+a_0(z),$ with $a_0(z)\neq 0$. Then $\nu(G(z))\leqslant \nu(a_0(z)).$ In particular, $\nu(G(z))\leqslant \deg_z P_G(z,y)$.
\end{lemma}

\begin{proof} Since $P_G(z,y)$ is a minimal polynomial, we have $a_0(z)\neq 0$. We thus have, identically, $$\left(a_n(z)G(z)^{n-1}+a_{n-1}(z)G(z)^{n-2}+\cdots+a_1(z)\right) G(z)=-a_0(z).$$ The fact $G(z),a_n(z),\ldots,a_0(z)\in\B{C}[[z]]$ then gives $$\nu\left(a_n(z)G(z)^{n-1}+a_{n-1}(z)G(z)^{n-2}+\cdots+a_1(z)\right)+\nu( G(z))=\nu(a_0(z)),$$ which proves the lemma, since each of the terms is a nonnegative integer.
\end{proof}

\begin{proof}[Proof of Theorem \ref{main}] Let $F(z)$ be a $k$-Mahler function satisfying \eqref{MFE} of degree $d_F$ and height $A_F$ and let $G(z)$ be an algebraic function of degree at most $n$ and height at most $H_G$. Since by Lemma \ref{BC2016}, the theorem holds for $n=1$, we may assume without loss of generality that $n\geqslant 2$. 

Set $M:=\nu(F(z)-G(z))$, and write $$F(z)-G(z)=z^{M}T(z),$$ where $T(z)\in\B{C}[[z]]$ with $T(0)\neq 0$. Then also $$\sum_{i=0}^d a_i(z)F(z^{k^i})-\sum_{i=0}^d a_i(z)G(z^{k^i})=\sum_{i=0}^d a_i(z)z^{k^iM}T(z^{k^i}),$$ which since $F(z)$ satisfies \eqref{MFE} reduces to $$M_G(z):=\sum_{i=0}^d a_i(z)G(z^{k^i})=-\sum_{i=0}^d a_i(z)z^{k^iM}T(z^{k^i}).$$ This immediately implies that $$\nu(F(z)-G(z))=M\leqslant \nu\left(M_G(z)\right)\leqslant \delta_{M_G},$$ where the last inequality follows from Lemma \ref{nuG}. By definition, $\delta_G=\log H_G$, hence applying Lemma \ref{MGprops} proves the theorem.\end{proof}

\section{Concluding remark}

The $n$-dependence in the estimate of Theorem \ref{main} is the best that can be attained by this method, that is, a bound of $n$-order $n^{d_F}$; this is the same $n$-order for the best known bounds on $\nu(Q(z,F(z)))$ as well \cite{T1998}. While at first glance, the $n$-dependence in Theorem \ref{main} looks like $n^{d_F+1}$, when using the results one usually first takes a limit through the height $H_G$. With this in mind, one assumes that $\log H_G\geqslant n\geqslant 1$ so that our estimate gives $$\nu\left(F(z)-G(z)\right)\leqslant \left((d_F+1)\cdot A_f+\frac{k^{d_F+1}-1}{k-1}\right)\cdot\log H_G\cdot n^{d_F}.$$

The immediate question is whether or not this is the best bound possible; probably the answer is `no.' Presumably a `Roth-type' estimate holds, so that one has a bound that is linear in $n$. This would imply that a Mahler function $F(z)$ is an $S$-number in the suitable function-field analogue of Mahler's classification (see Bugeaud \cite{B2004} for the relevant definitions), which is a question that has been circulating within the area for some time now.

\bibliographystyle{amsplain}

\begin{thebibliography}{10}

\bibitem{AS1992}
Jean-Paul Allouche and Jeffrey Shallit, \emph{The ring of {$k$}-regular
  sequences}, Theoret. Comput. Sci. \textbf{98} (1992), no.~2, 163--197.
  \MR{1166363}

\bibitem{ASbook}
\bysame, \emph{Automatic sequences}, Cambridge University Press, Cambridge,
  2003, Theory, applications, generalizations. \MR{1997038}

\bibitem{AS2003}
\bysame, \emph{The ring of {$k$}-regular sequences. {II}}, Theoret. Comput.
  Sci. \textbf{307} (2003), no.~1, 3--29, Words. \MR{2014728}

\bibitem{B1991}
Paul-Georg Becker, \emph{Effective measures for algebraic independence of the
  values of {M}ahler type functions}, Acta Arith. \textbf{58} (1991), no.~3,
  239--250. \MR{1121085}

\bibitem{pgB1994}
\bysame, \emph{{$k$}-regular power series and {M}ahler-type functional
  equations}, J. Number Theory \textbf{49} (1994), no.~3, 269--286. \MR{1307967
  (96b:11026)}

\bibitem{BC2016}
Jason~P. Bell and Michael Coons, \emph{Transcendence tests for mahler
  functions}, Proc. Amer. Math. Soc., to appear.

\bibitem{B2004}
Yann Bugeaud, \emph{Approximation by algebraic numbers}, Cambridge Tracts in
  Mathematics, vol. 160, Cambridge University Press, Cambridge, 2004.
  \MR{2136100}

\bibitem{CKMR1980}
G.~Christol, T.~Kamae, M.~Mend{\`e}s~France, and G.~Rauzy, \emph{Suites
  alg\'ebriques, automates et substitutions}, Bull. Soc. Math. France
  \textbf{108} (1980), no.~4, 401--419. \MR{614317}

\bibitem{DMP1982}
Michel Dekking, Michel Mend{\`e}s~France, and Alf van~der Poorten,
  \emph{Folds}, Math. Intelligencer \textbf{4} (1982), no.~3, 130--138.
  \MR{684028}

\bibitem{L1988}
J.~H. Loxton, \emph{Automata and transcendence}, New advances in transcendence
  theory ({D}urham, 1986), Cambridge Univ. Press, Cambridge, 1988,
  pp.~215--228. \MR{972002}

\bibitem{M1929}
K.~Mahler, \emph{Arithmetische {E}igenschaften der {L}\"osungen einer {K}lasse
  von {F}unktionalgleichungen}, Math. Ann. \textbf{101} (1929), no.~1,
  342--366.

\bibitem{Nes1977}
Ju.~V. Nesterenko, \emph{Estimate of the orders of the zeroes of functions of a
  certain class, and their application in the theory of transcendental
  numbers}, Izv. Akad. Nauk SSSR Ser. Mat. \textbf{41} (1977), no.~2, 253--284,
  477. \MR{0491535}

\bibitem{Nes1984}
Yu.~V. Nesterenko, \emph{Algebraic independence of algebraic powers of
  algebraic numbers}, Mat. Sb. (N.S.) \textbf{123(165)} (1984), no.~4,
  435--459. \MR{740672}

\bibitem{N1991}
Kumiko Nishioka, \emph{Algebraic independence measures of the values of
  {M}ahler functions}, J. Reine Angew. Math. \textbf{420} (1991), 203--214.
  \MR{1124571}

\bibitem{T1995}
Thomas T{\"o}pfer, \emph{Algebraic independence of the values of generalized
  {M}ahler functions}, Acta Arith. \textbf{70} (1995), no.~2, 161--181.
  \MR{1322559}

\bibitem{T1998}
\bysame, \emph{Zero order estimates for functions satisfying generalized
  functional equations of {M}ahler type}, Acta Arith. \textbf{85} (1998),
  no.~1, 1--12. \MR{1623365}

\end{thebibliography}
\def\polhk#1{\setbox0=\hbox{#1}{\ooalign{\hidewidth
  \lower1.5ex\hbox{`}\hidewidth\crcr\unhbox0}}} \def\cprime{$'$}
\providecommand{\bysame}{\leavevmode\hbox to3em{\hrulefill}\thinspace}
\providecommand{\MR}{\relax\ifhmode\unskip\space\fi MR }
\providecommand{\MRhref}[2]{%
  \href{http://www.ams.org/mathscinet-getitem?mr=#1}{#2}
}
\providecommand{\href}[2]{#2}


\end{document}